\documentclass[a4paper]{amsart}
\usepackage[T1]{fontenc}
\usepackage[french,english]{babel}
\usepackage{newlfont}
\usepackage{amsmath,amssymb,amsfonts,amsbsy,bbm}
\usepackage[all]{xy}

\usepackage{mathtools,braket,mathrsfs}
\usepackage{tikz}
\usepackage{tikz-cd}
\usetikzlibrary{arrows}
\usepackage{enumerate}
\usepackage{extarrows}

\newcommand{\Z}{\ensuremath{\mathbb{Z}}}
\newcommand{\Q}{\ensuremath{\mathbb{Q}}}
\newcommand{\C}{\ensuremath{\mathbb{C}}}
\newcommand{\R}{\ensuremath{\mathbb{R}}}
\newcommand{\Proj}{\ensuremath{\mathbb{P}}}		
\newcommand{\DD}{\ensuremath{\mathbb{D}}}

\newcommand{\too}{\longrightarrow}								%Arrows
\newcommand{\mapstoo}{\longmapsto}
\newcommand{\into}{\hookrightarrow}
\newcommand{\onto}{\twoheadrightarrow}

\DeclareRobustCommand\ontoo{\relbar\joinrel\twoheadrightarrow}

\newcommand{\spe}{\mathrm{sp}}
\newcommand{\ratq}{\mathrm{rq}}
\newcommand{\KM}{\mathrm{KM}}
\newcommand{\cZ}{\mathcal{Z}}
\newcommand{\cC}{\mathcal{C}}
\newcommand{\cO}{\mathcal{O}}
\newcommand{\cP}{\mathcal{P}}
\newcommand{\cB}{\mathcal{B}}
\newcommand{\sD}{\mathscr{D}}

\DeclareMathOperator{\CH}{CH}
\DeclareMathOperator{\HH}{H}

\DeclareMathOperator{\Ker}{ker}
\DeclareMathOperator{\Coker}{coker}
\DeclareMathOperator{\Span}{span}
\DeclareMathOperator{\Tr}{Tr}
\DeclareMathOperator{\Norm}{N}

\DeclareMathOperator{\Ind}{Ind}

\DeclareMathOperator{\Ort}{O}
\DeclareMathOperator{\SO}{SO}

\DeclareMathOperator{\Sym}{Sym}
\DeclareMathOperator{\M}{M}

\renewcommand{\det}{\operatorname{det}}
\renewcommand{\div}{\operatorname{div}}

\DeclareMathOperator{\weight}{wgt}

\def\XXint#1#2#3{{\setbox0=\hbox{$#1{#2#3}{\int}$ }
\vcenter{\hbox{$#2#3$ }}\kern-.585\wd0}}

\theoremstyle{plain}
\newtheorem{theorem}{Theorem}[section]
\newtheorem{lemma}[theorem]{Lemma}
\newtheorem{proposition}[theorem]{Proposition}
\newtheorem{corollary}[theorem]{Corollary}

\theoremstyle{definition}
\newtheorem{definition}[theorem]{Definition}

\newtheorem{question}{Question}

\theoremstyle{remark}
\newtheorem{remark}[theorem]{Remark}
\newtheorem{example}[theorem]{Example}

\makeatletter
\newenvironment{abstracts}{%
  \ifx\maketitle\relax
    \ClassWarning{\@classname}{Abstract should precede
      \protect\maketitle\space in AMS document classes; reported}%
  \fi
  \global\setbox\abstractbox=\vtop \bgroup
    \normalfont\Small
    \list{}{\labelwidth\z@
      \leftmargin3pc \rightmargin\leftmargin
      \listparindent\normalparindent \itemindent\z@
      \parsep\z@ \@plus\p@
      
      \itemsep\medskipamount
    }%
}{%
  \endlist\egroup
  \ifx\@setabstract\relax \@setabstracta \fi
}

\newcommand{\abstractin}[1]{%
  \otherlanguage{#1}%
  \item[\hskip\labelsep\scshape\abstractname.]%
}
\makeatother

\begin{document}

\title{On rational quadratic cocycles}

%\dedicatory{To Henri Darmon, on his 60th birthday}

\author{Lennart Gehrmann}
\address{L.~Gehrmann, Universität Bielefeld, Bielefeld, Germany.}
\email{gehrmann.math@gmail.com}

%\classification{11F75, } 

\begin{abstracts}
\abstractin{english}
Let $(V,q)$ be a non-degenerate $n$-dimensional quadratic space over the rationals of real signature $(r,s)$.
For every integer $1\leq k \leq \min\{r,n-2\}$ we construct classes in the cohomology of arithmetic subgroups of $\Ort(V)$ with values in the group of codimension $k$ cycles on the quadric of isotropic lines in $V$.
Generating series of images of these classes in an equivariant version of the $k$-th Chow group are shown to be Siegel modular forms of genus $k$ in the extremal cases $k=1$ and $k=r$.

\abstractin{french}
Soit $(V,q)$ un espace quadratique non dégénéré de dimension $n$ sur les rationnels, de signature réelle $(r,s)$.
Pour tout entier $1 \leq k \leq \min\{r,n-2\}$, nous construisons des classes dans la cohomologie des sous-groupes arithmétiques de $\Ort(V)$ à valeurs dans le groupe des cycles de codimension $k$ sur la quadrique des droites isotropes dans $V$.
Les séries génératrices des images de ces classes dans une version équivariante du $k$-ième groupe de Chow sont des formes modulaires de Siegel de genre $k$ dans les cas extrémaux $k=1$ et $k=r$.
\end{abstracts}

\maketitle

\tableofcontents

%%%%%%%%%%%%%%%%%%%%%%%%%%%%%%%%%%%%%%%%%%%%%%%%%%%%%%
% Introduction
%%%%%%%%%%%%%%%%%%%%%%%%%%%%%%%%%%%%%%%%%%%%%%%%%%%%%%
\section{Introduction}
Let $(V,q)$ be an $n$-dimensional non-degenerate quadratic space over the field of rational numbers of real signature $(r,s)$. 
Rigid meromorphic cocycles as introduced by Darmon--Vonk in \cite{DV} and later generalized to the orthogonal setting in \cite{DGL} are certain classes in the cohomology of $p$-arithmetic subgroups of the orthogonal group $\Ort(V)$ with values in non-zero meromorphic functions on an attached $p$-adic symmetric space $X_p$.
The $p$-adic symmetric space $X_p$ is an open rigid analytic subspace of the quadric $Q_V$ of isotropic lines in $V$.
The quadric carries a collection of divisors indexed by positive definite $1$-dimensional subspaces of $V$ called rational quadratic divisors.
The defining property of rigid meromorphic cocycles is that their divisors are given by some prescribed cohomology classes with values in the group of locally finite divisors on $X_p$ which are infinite linear combinations of rational quadratic divisors.
Similarly, all known constructions of rigid meromorphic cocycles involve infinite products of rational functions on $Q_V$ whose divisors are rational quadratic.
The aim of this note is to understand divisor-valued cohomology classes as well as their function-valued counterparts prior to taking $p$-adic limits.
Moreover, the theory is developed for rational quadratic cycles of arbitrary codimension.

\smallskip
\noindent \textbf{Overview of the paper.}
We begin by recalling some basic facts about special cycles on quadric hypersurfaces over arbitrary fields of characteristic different from $2$, which are given by quadrics associated with orthogonal complements of non-degenerate subspaces.
We show that their classes in the Chow group depend only on the dimension of the subspace.
As a corollary, we deduce that the subring of the Chow ring generated by special cycles is a truncated polynomial ring.

Starting with Section \ref{sec: equivariant} we consider the case of quadrics over $\Q$.
Here, we focus on rational quadratic cycles, that is, those special cycles which are attached to positive-definite subspaces.
Assuming that the quadratic space $V$ is anisotropic, the cohomology of arithmetic subgroups of $\Ort(V)$ with values in rational quadratic cycles of codimension $k$ is shown to be zero in degree less than $ks$.
Under the further assumption that the arithmetic subgroup $\Gamma$ is neat, we give a concrete description of the cohomology group in degree $ks$ that generalizes the calculation of divisor-valued cohomology groups in \cite{DV} and \cite{gehrmann-quaternionic} for $3$-dimensional quadratic spaces.
Equivariant versions of Chow groups of $Q_V$ together with a weight homomorphism to the cohomology of $\Gamma$ with trivial coefficients are defined.
The kernel of the weight map is shown to be torsion if $k=1$.
Simple considerations with cohomological dimensions show that the weight map is an isomorphism if $k=r$.

Rational quadratic cycles have an Archimedean counterpart, that is, cycles in the locally symmetric space attached to $\Ort(V)$.
We construct explicit $ks$-cocycles of $\Gamma$ with values in codimension $k$ cycles by blending rational quadratic cycles with their Archimedean counterparts, 
Under the weight homomorphism, they are mapped to the cohomology classes studied by Kudla and Millson (see \cite{KM}). 
The cohomology classes are indexed by positive definite $k\times k$-matrices and, thus, their attached generating series look like Fourier series of Siegel modular forms.
The main theorem of \cite{KM} implies that the image of these generating series under the weight homomorphism is in fact modular.
Applying the results of Section \ref{sec: equivariant} yields modularity of the generating in the equivariant Chow group if $k=1$ or $k=r$.

The study of rational quadratic cocycles in this note is far from complete and ,thus, we included several questions in the main body of the text.
There are also several topics that are not discussed at all, for example, the action of the orthogonal Hecke algebra, the connection to the work of Ash \cite{Ash} via logarithmic derivatives, and analogues of the intersection homomorphism for the split orthogonal group $\Ort(2,2)$ studied by Sprehe \cite{Sprehe}.

\smallskip
\noindent \textbf{Acknowledgements.}
It is my pleasure to thank Henri Darmon for introducing me to the wonderful and ever-expanding world of rigid meromorphic cocycles.
I thank Mike Daas, H\aa vard Damm-Johnsen, and Sören Sprehe for several inspiring discussions during our stay at Mathematisches Forschungsinstitut Oberwolfach.
I received funding from Deutsche Forschungsgemeinschaft (DFG, German Research Foundation) via the grant SFB-TRR 358/1 2023 -- 491392403.

%%%%%%%%%%%%%%%%%%%%%%%%%%%%%%%%%%%%%%%%%%%%%%%%%%%%%%%%%%%%%%%%%%%%%%%%
% Chow groups
%%%%%%%%%%%%%%%%%%%%%%%%%%%%%%%%%%%%%%%%%%%%%%%%%%%%%%%%%%%%%%%%%%%%%%%%

\section{Chow groups of quadrics}
In this section we describe special cycles on quadric hypersurfaces and their classes in the Chow ring.
To keep this section as self-contained as possible, we only consider cases relevant to the study of rational quadratic cocycles.

Throughout this section let $F$ be a field of characteristic $\mathrm{char}(F)\neq 2$.
\subsection{Chow groups}
Let us recall the definition and basic properties of Chow groups.
To that end, fix an equidimensional scheme $X/F$ of dimension $n$.
Write $\cZ^p(X)$ for the free abelian group on all closed, irreducible subsets of $X$ of codimension $p$.
Given such a subset $V\subseteq X$ we denote by $[V]$ the corresponding element of $\cZ_p(X)$.
More generally, if $V$ is a closed subscheme of $X$ of pure codimension $k$ we put
\[
[V]\coloneq \sum_{W}\ell(\mathcal{O}_{V,W}) [W] \in \cZ^k(X),
\]
where the sum is over the irreducible components of $V$.
Here $\mathcal{O}_{V,W}$ denotes the local ring of $\mathcal{O}_V$ at $W$ and $\ell(\mathcal{O}_{V,W})$ is its length.

Define $\cC^k(X)$ to be the free abelian group on pairs $(W,f)$, where $W\subset X$ is a closed, integral subvariety of codimension $k-1$ and $f\colon W \to \Proj^1_F$ is a non-zero rational function modulo the relations
\[
(W,f_1 f_2)= (W, f_1) + (W, f_2),
\]
where $f_1, f_2$ are non-zero rational functions on $W$.
Mapping a pair $(W,f)$ to the divisor of $f$ defines a homomorphism
\begin{equation}\label{eq: div}
\div\colon \cC^k(X)\too \cZ^k(X),\quad (W,f)\mapstoo \div(f).
\end{equation}
The \emph{Chow group} of codimension $k$ cycles on $X$ is defined as the cokernel
\[
\CH^k(X)\coloneq \Coker(\div\colon \cC^k(X)\too \cZ^k(X)).
\]
Given an element $\alpha \in \cZ^k(X)$ we write $\{\alpha\}$ for its class in $\CH^k(X)$.
The intersection product (see \cite{Fulton}, Chapter 6) endows $\CH^\ast(X)=\oplus_k \CH^k(X)$ with the structure of a commutative ring.
Its unit is given by the class of $[X]$.

Let $F'/F$ be an extension of fields.
Given a scheme $Y$ over $F$ we write $Y_{F'}$ for the fibre product $Y\times_{\mathrm{Spec}(F)}\mathrm{Spec}(F')$.
If $Y$ is integral, its base change $Y_{F'}$ is equidimensional.
Thus, for every integral scheme $X$, the homomorphism
\[
\cZ^k(X)\too \cZ^k(X_{F'}),\quad [V] \mapstoo [V_{F'}]
\]
is well-defined.
One easily checks that it descends to a homomorphism
\[
\iota_{F'}\colon\CH^k(X)\too \CH^k(X_F').
\]
In case $F'/F$ is a finite extension, the existence of proper push-forward maps on Chow groups yields a norm map in the opposite direction
\[
\Norm_{F'/F}\colon \CH^k(X_F') \too \CH^k(X).
\]
By construction, the composition $\Norm_{F'/F}\circ \iota_{F'}$ is equal to multiplication with the degree $[F' : F]$.
In particular, the kernel of $\iota_{F'}$ is torsion.

Write $|X|$ for the set of closed points of $X$.
For $x\in |X|$ let $\deg(x)$ be the degree of the residue field of $x$ over $F$.
Under the assumption that $X$ is proper, the homomorphism
\[
\cZ^n(X)\too \Z,\quad \sum_{x\in |X|} a_x [x] \mapstoo \sum_{x\in |X|} a_x \deg(x) \
\]
descends to a map
\[
\deg \colon \CH^n(X)\too \Z.
\]
For every extension $F'/F$ the diagram
\[\begin{tikzcd}
	{\CH^n(X)} & {} & \Z \\
	{\CH^n(X_{F'})} && \Z
	\arrow["\deg", from=1-1, to=1-3]
	\arrow[from=1-1, to=2-1]
	\arrow[from=1-3, to=2-3]
	\arrow["\deg", from=2-1, to=2-3]
\end{tikzcd}\]
is commutative.

\begin{example}
Consider the product $X=\Proj^1_F \times \Proj^1_F$ of two projective lines.
By the projective bundle formula (see \cite{Fulton}, Theorem 3.3) the first Chow group $\CH^{1}(X)$ is free of rank $2$ with generators given by the classes of $\Proj^1_F\times\{x\}$ and $\{x\}\times \Proj^1_F$ for any point $x\in \Proj^1(F)$.
Moreover, the degree map induces an isomorphism $\CH^2(X)\to \Z$.
\end{example}

\subsection{Special cycles on smooth quadrics}
Let $(V,q)$ be a finite-dimensional non-degenerate quadratic space over $F$ of dimension $n\geq 2$.
Denote by $\langle \cdot,\cdot \rangle$ the associated bilinear, that is, $q(v)=2\langle v,v\rangle$ for all $v\in V$.
If $F'/F$ is a field extension, we write $V_{F'}=V\otimes_F F'$ for the corresponding quadratic space over $F'$.
Given a subset $A\subseteq V$ let $A^{\perp}\coloneq \{ v\in V\, \vert\, \langle a,v\rangle=0\, \forall a\in A\}$ be its orthogonal complement.
Denote the quadric of isotropic lines in $V$ by
\[
Q_V\coloneq \{\ell\in \Proj(V)\, \vert\, q(\ell)=0\}.
\]
This is a smooth projective variety over $F$ of dimension $n-2$.
It is geometrically connected unless $n=2$, in which case $Q_V$ has two geometric connected components.
The action of the orthogonal group $\Ort(V)$ of $V$ on $Q_V$ induces an action on $\cZ^k(Q_V)$ that descends to an action on Chow groups.

Given a non-degenerate subspace $W\subseteq V$ of dimension $k\leq n-2$, define its associated \emph{special cycle}
\[
\Delta_{W}\coloneq [Q_{W^\perp}]\ \in \cZ^k(Q_V).
\]
Note that for two such subspaces $W,W'\subseteq V$ of dimension less than $n-1$, the cycles $\Delta_{W}$ and $\Delta_{W'}$ agree if and only if $W=W'$.
By definition, we have $g.\Delta_W=\Delta_{gW}$.
Write $\cZ^k(Q_V)_{\spe}\subseteq \cZ^k(Q_V)$ for the subgroup generated by all special cycles of codimension $k$.
It is the free abelian group on the $\Delta_W$, $W\subseteq V$ non-degenerate of dimension $k$.
Moreover, it is an $\Ort(V)$-submodule.
We call the $\Ort(V)$-equivariant homomorphism
\[
\weight\colon \cZ^k(Q_V)_{\spe}\longrightarrow \Z,\quad \sum_W a_W \Delta_W \mapstoo \sum_W {a_W}.
\]
the \emph{weight map}.
Note that in case $\dim(W)=n-2$, we have $\weight_{n-2}{\Delta_W}=2\deg(\Delta_W)$.
Further, \cite{Swan} has proven that the degree map for the top degree Chow group of $Q_V$ is injective.
It is surjective if and only if $Q_V(F)\neq \emptyset$.

Let $W,W'\subseteq V$ be non-degenerate subspaces that are orthogonal to each other.
Suppose that $\dim(W+W')\leq n-2$.
Then
\begin{equation}\label{eq: mult}
\{\Delta_{W}\}\cdot\{\Delta_{W'}\}=\{\Delta_{W+W'}\} \in \CH^{\ast}(Q_V).
\end{equation}
Given a non-degenerate subspace $W\subseteq V$ of dimension $i< n-2$, there exists a non-degenerate subspace $W'\subseteq V$ orthogonal to $W$ such that $\dim(W+W')=n-2$.
Since $\{\Delta_{W+W'}\}$ is non-torsion, \eqref{eq: mult} implies that $\{\Delta_{W}\}$ is non-torsion.
Suppose that $W_1,W_2\subseteq V$ are $1$-dimensional non-degenerate subspaces.
Fix generators $w_1,w_2$ of $W_1$ respectively $W_2$.
Then
\[
f_{w_1,w_2}(x)=\frac{\langle w_1, x \rangle}{\langle w_2, x \rangle}
\]
defines a rational function on $\Proj(V)$ and thus, by restriction, a rational function on the quadric $Q_V$.
By definition, we have
\[
\div(f_{w_1,w_2})=\Delta_{W_1} - \Delta_{W_2}.
\]
Therefore, the class of $\Delta_W$ in the Chow group does not depend on the choice of $1$-dimensional non-degenerate subspace.
By \eqref{eq: mult}, the same is true for the classes of higher-dimensional subspaces.

To summarize, let $\CH^k(Q_V)_{\spe}$ be the image of $\cZ^k(Q_V)_{\spe}$ in $\CH^k(Q_V)$.
We have proven that the weight map descends to an isomorphism
\[
\weight\colon \CH^k(Q_V)_{\spe}\too \Z
\]
of $\Ort(V)$-modules.
Furthermore, $\CH^\ast(Q_V)_{\spe}\coloneq \oplus_k \CH^k(Q_V)_{\spe}$ is a subring of the Chow ring of $Q_V$, which is isomorphic to $\Z[T]/T^{n-1}$.
It follows that for every field extension $F'/F$, the base change map $\CH^\ast(Q_{V})_{\spe}\to \CH^\ast(Q_{V_{F'}})_{\spe}$ is an isomorphism.
 
\begin{example}
Consider the space $V=M_{2}(F)$ of $2 \times 2$-matrices over $F$ together with $q=\det$ and the subspace $W\subseteq V$ of matrices of trace $0$.
An isotropic line $\ell\in Q_{V}$ is generated by a matrix $A$ of rank $1$.
The line is uniquely determined by the kernel and cokernel of $A$.
Thus, the assignment
\begin{equation}\label{iso: quadric}
Q_{V}\xlongrightarrow{\sim} \Proj^1_F \times \Proj^1_F,\quad \Span(A) \mapstoo \Ker(A) \times \Coker(A)
\end{equation}
defines an isomorphism of $F$-varieties,
which yields an identification
\[
\CH^{1}(Q_V)\cong \CH^{1}(\Proj^1_F \times \Proj^1_F)=\Z \times \Z.
\]
Similarly, an isotropic line in $Q_{W}$ is generated by a non-zero nilpotent matrix and the assignment
\[
Q_{W}\xlongrightarrow{\sim} \Proj^1_F ,\quad \Span(A) \mapstoo \Ker(A)
\]
defines an isomorphism of $F$-varieties.
In particular, the cycle $\Delta_{W^\perp}\in \cZ^1(Q_{V})$ is identified with the diagonal in $\Proj^1_F \times \Proj^1_F$.
Thus $\CH^1(Q_V)_{\spe}$ is equal to the diagonal in $\CH^{1}(Q_V)\cong \Z \times \Z$.
\end{example}

Let $\cC^k(Q_V)_{\spe}\subseteq \cC^k(Q_V)$ be the preimage of $\cZ^k(Q_V)_{\spe}$ under the divisor map.
Thus, there is a short exact sequence 
\[
\cC^k(Q_V)_{\spe} \xlongrightarrow{\div}\cZ^k(Q_V)_{\spe} \xlongrightarrow{\weight}\Z\too 0
\]
of $\Ort(V)$-modules.
Further, define $\cC^k(Q_V)_{\spe}^{+}\subseteq \cC^k(Q_V)_{\spe}$ to be the free abelian group on pairs $(Q_W,f)$ with $W\subseteq V$ of dimension $k-1$ and $f\colon \Delta_W\to \Proj^1_F$ such that $\div(f)\in \cZ^k(Q_V)_{\spe}$ modulo the relation
\[
(Q_W,f_1 f_2)=(Q_W,f_1)+(Q_W,f_2).
\]
In case $k=1$, we have $\cC^1(Q_V)_{\spe}^{+}=\cC^1(Q_V)_{\spe}$.

\begin{lemma}\label{lem: aniso}
Suppose that $V$ is anisotropic.
The sequence 
\[
\cC^k(Q_V)_{\spe}^{+} \xlongrightarrow{\div} \cZ^k(Q_V)_{\spe} \xlongrightarrow{\weight}\Z\too 0
\]
of $\Ort(V)$-modules is exact for all $0\leq k \leq n-2$.
\end{lemma}
\begin{proof}
First note that every subspace of $V$ is non-degenerate.
For given subspaces $W_1,W_2\subseteq V$ of dimension $k\leq n-2$, choose a sequence of subspaces
\[
W_1=U_0,U_1,\ldots, U_{t-1},U_t=W_2
\]
such that $\dim(U_i\cap U_{i+1})=k-1$ for every $0\leq i \leq t-1$.
Put $V_i\coloneq U_i\cap U_{i+1}$ and fix non-zero vectors $w_{i1},w_{i2} \in V_i$ orthogonal to $U_i$ respectively $U_{i+1}$.
The function $f_{w_{i1}, w_{i2}}$ is non-constant on the quadric $\Delta_{V_i}$.
The claim follows immediately from the equality
\[
\div(\Delta_{V_i}, f_{w_{i1}, w_{i2}})=\Delta_{U_i}-\Delta_{U_{i+1}}
\]
for all $0\leq i\leq t-1$.
\end{proof}

%%%%%%%%%%%%%%%%%%%%%%%%%%%%%%%%%%%%%%%%%%%%%%%%%%%%%%%%%%%%%%%%%%%%%%%%
% Equivariant
%%%%%%%%%%%%%%%%%%%%%%%%%%%%%%%%%%%%%%%%%%%%%%%%%%%%%%%%%%%%%%%%%%%%%%%%

\section{Cycle-valued cohomology groups}\label{sec: equivariant}
For the rest of this article, we further specialize to the case that $(V,q)$ is a non-degenerate quadratic space over the rationals of real signature $(r,s)$ and dimension $n=r+s\geq 3$.
We will study the cohomology of arithmetic subgroups of $\Ort(V)$ with values in groups of special cycles on $Q_V$ coming from positive definite subspaces.

\subsection{Equivariant cycles}
A special cycle $\Delta_{W}$ is called \emph{rational quadratic} if $W$ is positive-definite.
For every $1\leq k\leq \min\{r,n-2\}$ let $\cZ^k(Q_V)_{\ratq}\subseteq \cZ^k(Q_V)_{\spe}$ be the subgroup generated by all rational quadratic cycles.

Let us remind ourselves that an \emph{arithmetic subgroup} of $\Ort(V)$ is a subgroup that is commensurable to $\Ort(L)$, where $L\subseteq V$ is any even $\Z$-lattice, that is, $q(L)\subseteq \Z$.
A subgroup $\Gamma$ of $\Ort(V)$ is \emph{neat} if for every $\gamma\in\Gamma$ the subgroup of $\cC^\times$ generated by all eigenvalues of $\Gamma$ is torsion-free.
In the following we give some basic results on the cohomology of arithmetic subgroups of $\Ort(V)$ with values in $\cZ^k(Q_V)_{\ratq}$.
The analysis presented here is far from complete.
Its main purpose is to motivate the definition of the group of \emph{$\Gamma$-equivariant rational quadratic cycles} in Definition \ref{def: cycles} below. 

Given an integer $1\leq k \leq \min\{r,n-2\}$, write $\cP_k$ for the set of all $k$-dimensional positive-definite subspaces of $V$.
Then $\cZ^k(Q_V)_{\ratq}$ is just the free abelian group on the set $\cP_k$.
For any arithmetic subgroup $\Gamma\subseteq \Ort(V)$, we may decompose $\cP_K$ into $\Gamma$-orbits and decompose $\cZ^k(Q_V)_{\ratq}$ accordingly:
\begin{equation}\label{eq: VFL}
\cZ^k(Q_V)_{\ratq}= \bigoplus_{\mathcal{O}\in \Gamma\backslash \cP_k} \Z[\mathcal{O}].
\end{equation}
Since arithmetic groups are of type (VFL) (see \cite{BS}, Section 11.1), taking $\Gamma$-cohomology commutes with direct limits (see \cite{Se2}, p.~101).
In particular, the canonical homomorphism
\[
\bigoplus_{\mathcal{O}\in \Gamma\backslash \cP_k} \HH^{i}(\Gamma,\Z[\mathcal{O}])\too \HH^i(\Gamma, \cZ^k(Q_V)_{\ratq})
\]
is an isomorphism for every $i\geq 0$.
Choosing a representative $W$ of a $\Gamma$-orbit $\Gamma W=\mathcal{O}\in \Gamma\backslash \cP_k$ yields an isomorphism
\[
\Z[\mathcal{O}] \cong \Ind_{\Gamma_W}^{\Gamma} \Z
\]
of $\Gamma$-modules, where $\Gamma_W$ denotes the stabilizer of $W$ in $\Gamma$.
Recall that $\Ind_{\Gamma_W}^{\Gamma} \Z$ is simply the space of functions from $\Gamma/\Gamma_W$ to $\Z$ with finite support.
In particular, we see that the $\Gamma$-invariants of $\Ind_{\Gamma_W}^{\Gamma} \Z$ are non-zero if and only if the quotient $\Gamma/\Gamma_W$ is finite.
Reduction theory implies that this is only the case if and only if $V$ is positive definite.
In the positive definite case, the $\Gamma$-invariants of $\Ind_{\Gamma_W}^{\Gamma} \Z$ are simply given by the constant functions.
Thus, we have proven the following:
\begin{proposition}
Let $\Gamma\subseteq \Ort(V)$ be an arithmetic subgroup.
The $0$-th cohomology group $\HH^0(\Gamma,\cZ^k(Q_V)_{\ratq})$ is non-zero
if and only if $V$ is positive definite, in which case there is a canonical isomorphism
\[
\HH^0(\Gamma,\cZ^k(Q_V)_{\ratq}) \cong \bigoplus_{\Gamma \backslash\cP_k} \Z.
\]
\end{proposition}
One may use Bieri--Eckmann duality theory for arithmetic groups to generalize the proposition above to non-positive definite cases.
We will carry this out in the anisotropic case:
\begin{proposition}
Suppose that $V$ is anisotropic.
For every arithmetic subgroup $\Gamma\subseteq \Ort(V)$ we have
\[
\HH^i(\Gamma,\cZ^k(Q_V)_{\ratq}) =0 \quad \forall i< ks.
\]
If $\Gamma$ is furthermore neat, then $\HH^{ks}(\Gamma,\Z[\mathcal{O}])$ is infinite cyclic for all $\mathcal{O}\in \Gamma\backslash \cP_k$ and, thus, \eqref{eq: VFL} induces an isomorphism
\[
\HH^{ks}(\Gamma,\cZ^k(Q_V)_{\ratq})\cong \bigoplus_{\Gamma\backslash \cP_k} \Z.
\]
\end{proposition}
\begin{proof}
The statement for general arithmetic subgroups follows from the fact that every arithmetic subgroup has a normal, neat subgroup of finite index and the Hochschild--Serre spectral sequence.
Let us assume that $\Gamma$ is neat.
Since $V$ is anisotropic, $\Gamma$ is a duality group of cohomological dimension $rs$ with duality module $\Z$, that is, there are canonical isomorphisms
\[
\HH^i(\Gamma,M)\cong\HH_{rs-i}(\Gamma,M)
\]
for every $\Gamma$-module $M$.
Applying the duality isomorphism to $\Ind_{\Gamma_W}^{\Gamma}\Z$ yields
\[
\HH^i(\Gamma,\Ind_{\Gamma_W}^{\Gamma}\Z)\cong \HH_{rs-i}(\Gamma,\Ind_{\Gamma_W}^{\Gamma}\Z)\cong \HH_{rs-i}(\Gamma_W,\Z),
\]
where the second isomorphism follows from Shapiro's Lemma.
Since $\Gamma_W$ is a neat arithmetic subgroup of $\Ort(W)\times \Ort(W^\perp)$ with $W$ positive definite, it is a duality group of cohomological dimension $(r-k)s$ with duality module $\Z$.
It follows that
\[
\HH_{rs-i}(\Gamma_W,\Z)= 0\quad \forall i < ks
\]
and furthermore that
\[
\HH_{rs-ks}(\Gamma_W,\Z)\cong \HH^{0}(\Gamma_W,\Z)\cong \Z,
\]
which proves the claim.
\end{proof}

\begin{remark}
For general $V$, let $r_V$ be the dimension of a maximal isotropic subspace of $V$.
The integer $r_V$ is the rank of the orthogonal group $\Ort(V)$.
Therefore, by \cite{BS}, every neat arithmetic subgroup of $\Ort(V)$ is a duality group of cohomological dimension $rs-r_V$ and some duality module, which has infinite rank if $V$ is not anisotropic.
Note that $r_V\leq \min\{r,s\}$.
Moreover, it follows from Meyer's theorem that $r_V= s$ if $r\geq s+4$.
Arguing as in the proof of the proposition above thus yields that
\[
\HH^i(\Gamma,\cZ^k(Q_V)_{\ratq}) =0 \quad \forall i \leq ks
\]
under the assumption that $(r-k)\geq s+4$.
\end{remark}

\begin{question}
Let $\Gamma\subseteq \Ort(V)$ be an arithmetic subgroup.
Is $\HH^i(\Gamma,\cZ^k(Q_V)_{\ratq})$ zero or at least torsion for all $i\leq ks$?
\end{question}

\begin{definition}\label{def: cycles}
Let $1\leq k\leq \min\{r,n-2\}$ be an integer and $\Gamma\subseteq \Ort(V)$ an arithmetic subgroup.
We call
\[
\cZ^k(\Gamma\backslash Q_V)_{\ratq}\coloneq \HH^{ks}(\Gamma, \cZ^k(Q_V)_{\ratq})
\]
the group of \emph{$\Gamma$-equivariant rational quadratic cycles on $Q_V$ of codimension $k$}. 
\end{definition}

\begin{remark}\label{rem: topdim}
Let us consider the case $k=r\leq n-2$.
Assume that $V$ is not anisotropic.
Then the virtual cohomological dimension of every arithmetic subgroup $\Gamma\subseteq \Ort(V)$ is less than $rs$ and, therefore,
$\cZ^r(\Gamma\backslash Q_V)_{\ratq}$ is torsion.
\end{remark}

\subsection{Equivariant Chow groups}
Define $\cC^k(Q_V)_{\ratq}\subseteq \cC^k(Q_V)_{\spe}$ as the preimage of $\cZ^k(Q_V)_{\ratq}$ under the divisor map and  $\cC^k(Q_V)_{\ratq}^+\subseteq \cC^k(Q_V)_{\spe}^+$ to be the subgroup generated by those pairs $(W,f)$ with $W$ positive-definite and $\div(f)\in \cZ^k(Q_V)_{\ratq}^{+}$.
As in the case of special cycles, the sequence
\[
\cC^k(Q_V)_{\ratq} \xlongrightarrow{\div}\cZ^k(Q_V)_{\ratq} \xlongrightarrow{\weight}\Z\too 0
\]
is exact and $\cC^1(Q_V)_{\ratq}^+=\cC^1(Q_V)_{\ratq}$ is just the group of non-zero meromorphic functions on $Q_V$ with rational quadratic divisor.

\begin{lemma}\label{lem: anexact}
Assume that $V$ is anisotropic.
The sequence
\[
\cC^k(Q_V)_{\ratq}^+\xlongrightarrow{\div} \cZ^k(Q_V)_{\ratq}\xlongrightarrow{\weight}\Z \too 0
\]
of $\Ort(V)$-modules is exact for every integer $1\leq k \leq \min\{r,n-2\}$.
\end{lemma}
\begin{proof}
Let $W_1, W_2$ be two $k$-dimensional positive definite subspaces of $V$.
Arguing as in the proof of Lemma \ref{lem: aniso} it is enough to find a sequence of positive definite $k$-dimensional subspaces
\[
W_1=U_0,U_1,\ldots, U_{t-1},U_t=W_2
\]
such that $\dim(U_i\cap U_{i+1})=k-1$ for all $1\leq i \leq t-1$.
If $\dim(W_1\cap W_2)=k-1$, we are already done.
If not, let $U'_1\subseteq W_1$ be any $(k-1)$-dimensional subspace containing the intersection $W_1\cap W_2$ and put $V'_1=U'_1+W_2$.
Then $V'_1$ has real signature $(r',s')$ with $s'\geq k$.
Therefore, there exists a vector $\tilde{u}$ in the orthogonal complement $(U'_1)^\perp$ of $U'_1$ in $V'_1$ with $q(\tilde{u})>0$.
Since orthogonal projection maps $W_2$ surjectivity onto $(U'_1)^\perp$ we may choose a preimage $u\in W_2$ of $\tilde{u}$.
Let $U_1$ be the span of $U_1'$ and $u$, which is the same as the span of $U_1'$ and $\tilde{u}$.
It is a $k$-dimensional positive-definite subspace of $V$.
By construction, we have $\dim(W_1\cap U_1)=k-1$ and $\dim(W_2 \cap U_1)=\dim (W_1\cap U_1)+1$.
The claim thus follows by induction.
\end{proof}

One of our main aims is to analyze when certain $\Gamma$-equivariant rational quadratic cycles lift to cohomology classes with values in $\cC^k(Q_V)_{\ratq}$.
To that end, we put
\[
\cC^k(\Gamma\backslash Q_V)_{\ratq}\coloneq \HH^{ks}(\Gamma,\cC^k(Q_V)_{\ratq})
\]
for $1\leq k \leq \min\{r,n-2\}$ and define the \emph{k-th $\Gamma$-equivariant rational quadratic Chow group} as the cokernel
\[
\CH^k(\Gamma \backslash Q_V)_{\ratq}\coloneq \Coker(\div\colon \cC^k(\Gamma\backslash Q_V)_{\ratq}\to \cZ^k(\Gamma\backslash Q_V)_{\ratq}).
\]
Given a $\Gamma$-equivariant rational quadratic cycle $\sD\in \cZ^k(\Gamma\backslash Q_V)_{\ratq}$, write $\{\sD\}$ for its image in $\CH^k(\Gamma \backslash Q_V)_{\ratq}$.
The weight map induces a homomorphism
\begin{equation}\label{eq: weight}
\weight\colon \CH^k(\Gamma \backslash Q_V)_{\ratq} \too \HH^{ks}(\Gamma,\Z).
\end{equation}
In case $V$ is anisotropic, we refine the $\Gamma$-equivariant rational quadratic Chow group by putting
\[
\cC^k(\Gamma\backslash Q_V)_{\ratq}^+\coloneq \HH^{ks}(\Gamma,\cC^k(Q_V)_{\ratq}^+)
\]
and
\[
\CH^k(\Gamma \backslash Q_V)_{\ratq}^+\coloneq \Coker(\div\colon \cC^k(\Gamma\backslash Q_V)_{\ratq}^+\to \cZ^k(\Gamma\backslash Q_V)_{\ratq}).
\]
Note that there is a canonical surjection $\CH^k(\Gamma \backslash Q_V)_{\ratq}^+ \onto \CH^k(\Gamma \backslash Q_V)_{\ratq}$.
Given a class $\sD\in \cZ^k(\Gamma\backslash Q_V)_{\ratq}$, write $\{\sD\}^+$ for its image in $\CH^k(\Gamma \backslash Q_V)_{\ratq}^+$.

The main goal of this section is to show that the weight map \eqref{eq: weight} is essentially an injection in case $k=1$.
\begin{lemma}\label{lem: trivialize}
Let $\Gamma\subseteq\Ort(V)$ be a neat arithmetic subgroup and $W\subseteq V$ a positive-definite subspace.
The stabilizer of the subspace $W$ in $\Gamma$ acts pointwise trivially on $W$.
\end{lemma}
\begin{proof}
The stabilizer of $W$ in $\Gamma$ is an arithmetic subgroup of $\Ort(W)\times \Ort(W^\perp)$.
Its projection onto the first factor is an arithmetic subgroup of $\Ort(W)$, which is finite since $W$ is positive-definite.
By the neatness of $\Gamma$, it must be trivial.
\end{proof}

Write $\cZ^k(Q_V)_{\ratq}^0\subseteq \cZ^k(Q_V)_{\ratq}$ for the kernel of the weight map.
\begin{proposition}\label{prop: main}
Let $\Gamma\subseteq \Ort(V)$ be a neat arithmetic subgroup. 
There exists a $\Gamma$-invariant section of the divisor map
\[
\div\colon \cC^1(Q_V)_{\ratq} \ontoo \cZ^1(Q_V)_{\ratq}^0.
\]
\end{proposition}
\begin{proof}
By Lemma \ref{lem: trivialize}, we may choose for every positive-definite $1$-dimensional subspace $W\subseteq V$ a generator $v_W$ such that $v_{\gamma W}=\gamma v_{W}$ for all $\gamma \in \Gamma$.
The homomorphism
\[
\cZ^1(Q_V)_{\ratq}^0\too \cC^1(Q_V)_{\ratq},\quad \sum_{W} n_W  \Delta_{W}\mapstoo \prod_{W} \langle v_W,\cdot \rangle^{n_W},
\]
gives the desired section.
\end{proof}

\begin{corollary}\label{cor: trivial}
Let $\Gamma\subseteq\Ort(V)$ be an arithmetic subgroup.
For $k=1$, the kernel of the weight map
\[
\weight\colon \CH^1(\Gamma \backslash Q_V)_{\ratq} \too \HH^{s}(\Gamma,\Z)
\]
is torsion.
If $\Gamma$ is furthermore neat, then the kernel is trivial.
\end{corollary}
\begin{proof}
The neat case is an immediate consequence of Proposition \ref{prop: main}.
The general case follows from the fact that every arithmetic subgroup has a normal, neat subgroup of finite-index and the Hochschild--Serre spectral sequence.
\end{proof}

Corollary \ref{cor: trivial} implies, in particular, that $\CH^1(\Gamma \backslash Q_V)_{\ratq}\otimes \Q$ is finite dimensional for every arithmetic group $\Gamma\subseteq \Ort(V)$.
\begin{question}
Is $\CH^k(\Gamma \backslash Q_V)_{\ratq}\otimes \Q$ finite dimensional for all $1\leq k \leq \min\{r,n-2\}$?
Is $\CH^k(\Gamma \backslash Q_V)_{\ratq}^+\otimes \Q$ finite dimensional in case $V$ is anisotropic?
\end{question}
If $V$ is positive definite, then $\CH^k(\Gamma \backslash Q_V)_{\ratq}\otimes \Q=(\CH^k(Q_V)\otimes \Q)^\Gamma$ is at most one-dimensional.
We give a few positive answers to the question above in arbitrary signature that follow purely from cohomological dimension considerations:
for that assume that $k=r\leq n-2$.
In case $V$ is not anisotropic, the group $\cZ^r(\Gamma \backslash Q_V)_{\ratq}$ is torsion for every arithmetic subgroup $\Gamma\subseteq \Ort(V)$ by Remark \ref{rem: topdim} and, hence, $\CH^r(\Gamma \backslash Q_V)_{\ratq}\otimes \Q=0$.
In case $V$ is anisotropic, the cohomological dimension of any neat arithmetic subgroup $\Gamma$ is equal to $rs$ and, therefore, the functor $\HH^{rs}(\Gamma,\cdot)$ is right exact.
It follows that the weight map induces isomorphisms
\begin{equation}\label{eq: weightisom}
\CH^r(\Gamma \backslash Q_V)_{\ratq}^+\xlongrightarrow{\sim}\CH^r(\Gamma \backslash Q_V)_{\ratq}\xlongrightarrow{\sim}\HH^{rs}(\Gamma,\Z)\cong \Z
\end{equation}
assuming that $\Gamma$ is neat.
For general arithmetic subgroups, it still follows that the kernel of the weight map is torsion by a Hochschild--Serre argument.

%%%%%%%%%%%%%%%%%%%%%%%%%%%%%%%%%%%%%%%%%%%%%%%%%%%%%%%%%%%%%%%%%%%%%%%%
% Generating series
%%%%%%%%%%%%%%%%%%%%%%%%%%%%%%%%%%%%%%%%%%%%%%%%%%%%%%%%%%%%%%%%%%%%%%%%

\section{Generating series of rational quadratic cocycles}
In this section, we will define generating series of cohomology classes with values in rational quadratic cycles and study their modularity properties.
The singular chain complex of a topological space $X$ will be denoted by $(C_\bullet(X), d_\bullet)$ and the singular homology groups by $\HH_\bullet(X)$.
Given an abelian group $M$ and an integer $i\in \Z$, write $M[i]$ for the complex concentrated in degree $i$ with $i$-th entry equal to $M$.

\subsection{Intersection numbers}
Let $X_\infty\coloneq X_\infty(V)$ be the space of maximal (i.e., $s$-dimensional) negative-definite subspaces of $V_\R\coloneq V\otimes \R$.
This is a model for the symmetric space of the special orthogonal group $\SO(V_\R)$.
In particular, it is a contractible smooth manifold of dimension $rs$.
As explained in \cite[p.~130]{KM} an orientation of $V$ canonically induces an orientation of $X_\infty$.
We hence fix an orientation of $V$ once and for all.

Given a positive-definite subspace $W\subseteq V$ of dimension $k$, we consider the special topological cycle
\[
\Delta_{W,\infty}\coloneq\{Z\in X_\infty\ \vert\ Z\mbox{ is orthogonal to } W\}.
\]
The special cycle $\Delta_{W,\infty}$ is naturally identified with the symmetric space $X_\infty(W^\perp)$.
Thus it is a smooth contractible manifold of codimension $ks$.
If $\underline{w}=(w_1,\ldots,w_k)$ is an ordered basis of $W$, then
\begin{equation}\label{eq: cyc intersection}
\Delta_{W,\infty}=\bigcap_{i=1}^k \Delta_{\Span(w_i),\infty}.
\end{equation}
Moreover, $\underline{w}$ induces an orientation of $\underline{W}$, which in turn yields an orientation of the associated cycle (see \cite[p.~131]{KM}).
We denote this oriented cycle by $\Delta_{\underline{w},\infty}$.

\begin{proposition}
Let $W\subseteq V$ be a positive-definite subspace of dimension $k$.
Then:
\[
\widetilde{\HH}_q( X_\infty \setminus \Delta_W)\cong
\begin{cases}
\Z &\mbox{if}\ q=ks-1\\
0 & \mbox{otherwise.}
\end{cases}
\]
A choice of orientation of $W$ gives an identification of $\widetilde{\HH}_{ks-1}( X_\infty \setminus \Delta_W)$ with $\Z$.
\end{proposition}
\begin{proof}
This is a consequence of the fact that $\Delta_W$ is a totally geodesic submanifold of $X_\infty$ with respect to the standard Riemannian metric on $X_\infty$ (see \cite[Section 2.4.4]{DGL} for a detailed discussion in case $k=1$).
\end{proof}
Let $\SO(V)^+\subseteq \SO(V)$ be the subgroup of elements with trivial real spinor norm.
Then every element of the spin group $\SO(V)^+$ that fixes $W$ and acts orientation-preservingly on $W$ acts trivially on $\widetilde{\HH}_{ks-1}( X_\infty \setminus \Delta_W)$.

Let $W\subseteq V$ be a $k$-dimensional positive-definite subspace with fixed basis $\underline{w}$ and $c\in C_{ks}(X_\infty)$ a singular chain with $d_{ks-1}(c)\in C_{ks-1}(X_\infty\setminus \Delta_{W,\infty})$.
The \emph{signed intersection number} $\Delta_{\underline{w},\infty}\cdot c\in \Z$ of $\Delta_{\underline{w},\infty}$ and $c$ is defined via
\[
\Delta_{\underline{w},\infty}\cdot c \coloneq [d_{ks}(c)] \in \tilde{\HH}_{ks-1}(X_\infty\setminus \Delta_{W,\infty})=\Z,
\]
where the last equality is induced by the orientation of $W$ coming from $\underline{w}$.
The equality
\[
\Delta_{g\underline{w},\infty}\cdot g c= \Delta_{\underline{w},\infty}\cdot c
\]
holds for every $g\in \SO(V)^+$.

\subsection{Cohomology classes}
Recall that for $1\leq k\leq {n-2}$, the set of all $k$-dimensional positive-definite subspaces of $V$ is denoted by $\cP_k$. 
Define the $\Ort(V)$-stable subcomplex $\mathfrak{C}_\bullet^{(k)}\subseteq C_\bullet(X_\infty)$ by
\[
\mathfrak{C}_q^{(k)}\coloneq
\begin{cases}C_{q}(X_\infty) & \mbox{ for } q>ks,\\
C_{q}(X_\infty \setminus \bigcup_{W\in \cP_k}\Delta_{W,\infty})& \mbox{ for } q< ks,
\end{cases}
\]
\[
\mathfrak{C}^{(k)}_{ks}=\left\{c\in C_{ks}(X_\infty)\ \middle\vert\  d_{ks}(c)\in C_{ks-1}\left(X_\infty \setminus \bigcup_{W\in \mathcal{P}_k}\Delta_{W,\infty}\right)\right\}.
\]
The same arguments as in \cite[Section 2.4.1]{DGL} imply that the inclusion $\mathfrak{C}_\bullet^{(k)}\into C_\bullet(X_\infty)$ is a quasi-isomorphism.
In particular, $\mathfrak{C}_\bullet^{(k)}$ is a resolution of the trivial $\Ort(V)$-module $\Z$.

Write $\cB_k$ for the set of oriented bases of $k$-dimensional positive-definite subspaces of $V$.
A subset $\cO\subseteq\cB_k$ is called \emph{locally finite} if for all compact subsets $K\subseteq X_\infty$ the set $\{\underline{w}\in \cO\ \vert\ \Delta_{\underline{w}}\cap K \neq \emptyset\}$ has finite cardinality.
For every locally finite subset $\cO\subseteq\cB_k$ the assignment
\[
\mathfrak{C}^{(k)}_{ks} \too \Z,\quad c\mapstoo \sum_{\underline{w}\in\cO} \Delta_{\underline{w},\infty}\cdot c, 
\]
is well-defined and yields a homomorphism of complexes
\[
c^{\KM}_{\cO}\colon \mathfrak{C}^{(k)}_{\bullet}\too \Z[ks].
\]
Similarly, the assignment
\[
\mathfrak{C}^{(k)}_{ks} \too \cZ^k(Q_V)_{\ratq},\quad c\mapstoo \sum_{\underline{w}\in\cO} (\Delta_{\underline{w},\infty}\cdot c)\ \Delta_{\underline{w}}, 
\]
where $\Delta_{\underline{w}}\coloneq \Delta_{\Span(\underline{w})}\in \cZ^k(Q_V)_{\ratq}$,
defines a homomorphism of complexes
\[
c^{\ratq}_{\cO}\colon \mathfrak{C}^{(k)}_{\bullet}\too \cZ^k(Q_V)_{\ratq}[ks].
\]
Be definition, we have
\begin{equation}\label{eq: comp cocycles}
c^{\KM}_{\cO}=\weight\circ\,  c^{\ratq}_{\cO}
\end{equation}
If a subgroup $\Gamma\subseteq \SO(V)^+$ stabilizes $\cO$, then both, $c^{\KM}_{\cO}$ and $c^{\ratq}_{\cO}$, are $\Gamma$-equivariant,
Write $\sD^{\KM}_{\cO}$ respectively $\sD^{\ratq}_{\cO}$ for the corresponding classes in $\HH^{ks}(\Gamma, \Z)$ respectively $\HH^{ks}(\Gamma, \cZ^k(Q_V)_{\ratq})$ (see for example \cite[Section 2.1]{DGL} for the construction of the associated cohomology classes).
Equation \eqref{eq: comp cocycles} immediately implies that
\begin{equation}\label{eq: comp classes}
\weight\hspace{-0,2em}\left(\sD^{\ratq}_{\cO}\right)=\sD^{\KM}_{\cO}.
\end{equation}

\subsection{Kudla--Millson theta series}
For $\underline{w}\in \mathcal{B}_k$ write $G(\underline{w})$ for the associated Gram matrix, that is, $G(\underline{w})_{ij}=\langle w_i,w_j\rangle$.
Fix an even $\Z$-lattice $L\subseteq V$.
Let $L^\sharp$ be its dual lattice with respect to $\langle \cdot,\cdot\rangle$ and write
\[
\DD_L\coloneq L^\sharp/L
\]
for its discriminant module.
For a commutative ring $R$ denote by $\Sym_k(R)\subseteq \M_k(R)$ the $R$-module of symmetric $k\times k$-matrices.
Let $\beta$ be an element of $\Sym_k(\R)$.
Write $\beta > 0$ if $\beta$ is positive definite and $\beta \geq 0$ if $\beta$ is positive semi-definite.
Given a tuple of cosets $h=(h_1,\ldots,h_k)\in \DD_L^k$ and a matrix $\beta\in \Sym_k(\Q), \beta >0$ define
\[
\cO(\beta,h)\coloneq\{\underline{w}\in \cB_k\ \vert\ G(\underline{w})=\beta,\, w_i\in h_i\ \forall 1\leq i\leq k\}.
\]

\begin{lemma}
The subset $\cO(\beta,h)\subseteq \cB_k$ is locally finite.
\end{lemma}
\begin{proof}
By \cite[Lemma 2.6]{DGL}, the set of all $\cO(\beta_{ii}, h_i)\subseteq \cB_1$ intersecting a given compact subset of $X_\infty$ is contained in a compact subset of $V_\R$.
Since $h_i$ is discrete, it follows that $\cO(\beta_{ii}, h_i)$ is locally finite for every $1\leq i\leq k$.
The assertion thus follows from \eqref{eq: cyc intersection}.
\end{proof}

Let $\beta,\ h$ be as before and write $\Gamma_h\subseteq \SO(V)^+$ for the stabilizer of $h=(h_1,\ldots,h_k)$ in $\SO(V)^+$, that is, the intersection of the stabilizers of $h_i$, $1\leq i \leq k$.
Note that $\Gamma_h$ is an arithmetic subgroup.
The class
\[
\sD^{\KM}_{\beta,h}\coloneq \sD^{\KM}_{\cO(\beta,h)}\in \HH^{ks}(\Gamma_h,\Z)
\]
is the \emph{Kudla--Millson class} associated to $\beta$ and $h$.
The class
\[
\sD^{\ratq}_{\beta,h}\coloneq \sD^{\ratq}_{\cO(\beta,h)}\in \HH^{ks}(\Gamma_h, \cZ^k(Q_V)_{\ratq})
\]
is the \emph{rational quadratic class} associated to $\beta$ and $h$.

Recall that a Siegel modular form $F$ of genus $g$ has a Fourier series expression of the form
\[
F(\tau)=\sum_{\beta \geq 0} a_\beta\, q^{\beta},\quad a_\beta\in \C,
\]
where $q^\beta$ is shorthand for $e^{2\pi i\Tr(\beta \tau)}$ and $\beta$ runs through the set of all positive semi-definite symmetric matrices in $\Sym_g(\Q)$.
We call the Fourier series
\[
F(\tau)=\sum_{\beta >0 } a_\beta
\]
where $\beta$ runs through all positive definite matrices the \emph{non-singular part} of the modular form $F$.
Let $A$ be an abelian group.
We say that a formal expression
\[
F=\sum_{\beta\geq 0} a_\beta\, q^{\beta},\quad a_\beta\in A
\]
with $\beta$ as above is \emph{(the non-singular part of) a Siegel modular form of genus $g$ and weight $k$ with values in $A$} if the formal series
\[
\varphi(F)\coloneq \sum_{\beta} \varphi(a_\beta)\, q^{\beta}
\] 
is (the non-singular part of) the Fourier expansion of a Siegel modular form of genus $g$ and weight $k$ for every homomorphism $\varphi\colon A \to \C$.
If the genus is equal to $1$, we simply say that $F$ is (the non-singular part of) a modular form of weight $k$.
In that case $\beta$ simply runs through the set of positive rational numbers.
The following modularity result is the main theorem of \cite{KM}.
\begin{theorem}\label{thm: KM}
Let $1\leq k \leq r$ and $h\in \DD_L^k$.
The formal series
\[
\Theta^{\KM}_{h}\coloneq \sum_{\beta > 0 } \sD^{\KM}_{\beta,h}\, q^\beta
\]
is the the non-singular part of a Siegel modular form of genus $k$ and weight $n/2$ with values in $\HH^{ks}(\Gamma_h,\Z)$.
If $s$ is odd, then $\Theta^{\KM}_h$ is a Siegel modular form.
\end{theorem}
Note that, if $V$ is positive definite, then $\Theta^{\KM}_h$ is just a classical Siegel theta series of a (coset of a) positive definite lattice.
In \cite{KM}, classes $\sD^{\KM}_{\beta,h}$ are constructed for $\beta$ semi-definite, and the formal $q$-series over all such $\beta$ is proven to be a Siegel modular form.
It is unclear how to promote the construction of $\sD^{\KM}_{\beta,h}$ for non-positive definite $\beta$ to rational quadratic classes.

For a coset $h\in \DD_L^k$ define the formal $q$-series
\[
\Theta^{\ratq}_h\coloneq \sum_{\beta >0} \{\sD^{\ratq}_{\beta,h}\}\, q^\beta
\]
with values in $\CH^k(\Gamma_h \backslash Q_V)_{\ratq}$.
In case $V$ is anisotropic, define
\[
\Theta^{\ratq,+}_h\coloneq \sum_{\beta >0} \{\sD^{\ratq}_{\beta,h}\}^+\, q^\beta
\]
In case $V$ is positive definite, modularity of Siegel theta series immediately implies that $\Theta^{\ratq}_h$ and $\Theta^{\ratq,+}_h$ are non-singular parts of Siegel modular forms of genus $k$ and weight $n/2$. 
Combining Theorem \ref{thm: KM} with Corollary \ref{cor: trivial} in the case $k=1$ or with the isomorphism \eqref{eq: weightisom} in the case $k=r$, respectively, yields the following modularity results for quadratic spaces of arbitrary signature:
\begin{theorem}\label{thm: ratq}
Assume that either $k=1$ or $k=r\leq n-2$.
For every coset $h\in \DD^k_L$ the formal $q$-series $\Theta^{\ratq,+}_h$ is the the non-singular part of a Siegel modular form of genus $k$ and weight $n/2$ with values in the $\Gamma_h$-equivariant rational Chow group.
If $s$ is odd, then $\Theta^{\ratq}_h$ is a cusp form.
In case $V$ is anisotropic, the same statements hold for $\Theta^{\ratq,+}_h$.
\end{theorem}

\begin{question}
Let $1\leq k \leq r-2$ and $h\in \DD_L^k$.
Is the generating series $\Theta^{\ratq}_h$ a Siegel modular form of genus $k$ and weight $n/2$?
In case $V$ is anisotropic, is the same true for $\Theta^{\ratq,+}_h?$
\end{question}
Note that the case $n=3$ is covered completely by Theorem \ref{thm: ratq}.
In dimension $4$, only the case of signature $(3,1)$ and $k=2$ remains open.

%%%%%%%%%%%%%%%%%%%%%%%%%%%%%%%%%%%%%%%%%%%%%%%%%%%%%%
% BIBLIOGRAPHY
%%%%%%%%%%%%%%%%%%%%%%%%%%%%%%%%%%%%%%%%%%%%%%%%%%%%%%

\bibliography{bibfile}
\bibliographystyle{alpha}

\end{document}